\documentclass[12pt]{amsart}

\usepackage{amssymb}

\usepackage{enumitem}

\usepackage{graphicx}

\makeatletter
\@namedef{subjclassname@2020}{%
  \textup{2020} Mathematics Subject Classification}
\makeatother

\usepackage[T1]{fontenc}

\newtheorem{theorem}{Theorem}[section]

\newtheorem{lemma}[theorem]{Lemma}

\theoremstyle{definition}
\newtheorem{definition}[theorem]{Definition}

\numberwithin{equation}{section}

\usepackage{graphicx} 

\usepackage{wrapfig}
\usepackage{amssymb}
\usepackage{dirtytalk}
\usepackage{float}
\usepackage[bb=dsserif]{mathalpha}
\usepackage{mathtools}
\usepackage{amsthm}
\usepackage[utf8]{inputenc}
\usepackage[T1]{fontenc}
\usepackage{lmodern}
\usepackage{amsfonts}
\usepackage{hyperref}
\usepackage{lipsum} 
\usepackage{geometry}
\usepackage{dcolumn}
\usepackage{amsmath}
\usepackage[nobysame]{amsrefs}
\newcommand\restr[2]{{
  \left.\kern-\nulldelimiterspace 
  #1 
  \vphantom{\big|} 
  \right|_{#2} 
  }}

\theoremstyle{definition}

\theoremstyle{remark}

\numberwithin{equation}{section}

\begin{document}

 \title[Sampling on infinite-dimensional PW spaces on graphs]{Sampling on Paley--Wiener spaces on graphs, with particular focus on the infinite-dimensional case}


\author[F.Giannoni]{Filippo Giannoni}
\address{Department of Mathematics, University of Sussex \\ Brighton, BN1 9QH \\ United Kingdom}
\email{F.Giannoni@sussex.ac.uk}











\begin{abstract}
We prove a sampling theorem for infinite-dimensional Paley–Wiener spaces on graphs which allows for stable frame reconstruction. We provide a generalization of the results presented in \cite{PW1}, where a frame reconstruction sampling theorem for finite graphs is presented. We prove that all sampling sets for the Paley--Wiener space $PW_\omega(G)$ are complements of $\lambda$-sets, thereby providing a sufficient condition for stable sampling and reconstruction on graphs such as $\mathbb{Z}^n$-lattices and trees of bounded geometry.
\end{abstract}
\maketitle
\section{Introduction}
Sampling via frame reconstruction on Paley--Wiener spaces on graphs was first introduced by I. Pesenson in \cite{PW1}. This paper has the valuable feature of taking into account very few assumptions on the graph's structure, thus making the results valid for a variety of graphs of analytical interest, like $\mathbb{Z}^n$-lattices and homogeneous trees (see e.g. \cite{figa}). However, the reconstruction theorem proposed in \cite{PW1} was later revised in the Erratum \cite{PW2}, adding the crucial assumption of finite-dimensional Paley--Wiener spaces. This restriction causes a drawback on the relevance of the results, since the aforementioned graphs all have infinite-dimensional Paley--Wiener spaces. Our aim in this paper is to generalise the results presented in \cite{PW1} by finding a frame reconstruction theorem which holds for infinite-dimensional Paley--Wiener spaces and allows for stable reconstruction. We will rely on the definitions of uniqueness set (\ref{def uset}) and $\lambda$-set (\ref{def lset}). In particular, we will show that the complement of a $\lambda$-set is also a viable sampling set for certain Paley--Wiener spaces (Theorem \ref{belteorema}). We will also show that every sampling set is the complement of a $\lambda$-set (Theorem \ref{iff}) and we will give an estimate of the constant $\lambda$.

Our study is not the first to address the infinite-dimensional frame reconstruction sampling problem since the publication of \cite{PW2}; see, for example, \cite{FG}, \cite{FP}, \cite{P}, \cite{PP1}, and \cite{PP2}. Nevertheless, our results are distinct from those in the cited literature because we focus primarily on the interplay between $\mathbb{\lambda}$-sets and sampling sets. Other publications either provide an algorithm for constructing a sampling set and subsequently derive that its complement is a $\lambda$-set (\cite{FG}, \cite{FP}, \cite{PP1}), or they establish less sharp results in more restricted settings (\cite{P}, where connectedness is assumed and part 2 of Theorem 1.1 presents a less precise constant for graphs where the Laplacian operator is not invertible, such as finite graphs and $\mathbb{Z}^n$-lattices). We will discuss this topic in detail by comparing our results to those in \cite{FP}, and for the other cited works, similar arguments can be demonstrated. 

The secondary goal of this work is to provide a brief note on this topic, as the original results of \cite{PW1} are still being cited in the infinite-dimensional setting without acknowledgment of the erratum (see, e.g., \cite{MM}, where Section 5 contains a comparison between the sampling set found in that paper and one provided by an example from \cite{PW1}. Since those results are presented for the square lattice, they should not be taken into account, as their validity had remained unproven until now).

The paper is organized as follows: Section 2.1 is a brief overview of the known sampling results in the continuous case, Section 2.2 contains the main points of Pesenson's work, along with comments on why Theorem \ref{teoPes} is problematic. Section 3 provides the sampling results for finite and infinite-dimensional Paley--Wiener spaces we are looking for (Theorems \ref{belteorema} and \ref{iff}). Section 4 contains a thorough discussion on \cite{FP}. First, we will highlight the difference in the methodology (and thus in the results) with respect to our work. Then, we will provide an example of a sampling set which is valid in both cases, and show how our work leads to a sharper constraint, which allows for more Paley--Wiener spaces to be sampled on the same set. Finally, section 5 will be devoted to finding notable sampling examples on homogeneous trees.

\section{Overview and motivations}
\subsection{Results in the continuous case}
Sampling theory on the real line $\mathbb{R}$ has been developed over the last century and it relies heavily on Paley--Wiener spaces, which are defined through the Fourier transform on $L^2(\mathbb{R})$. Recall that the Fourier transform $\mathcal{F}$ of a function $f \in L^1(\mathbb{R)}$ is
$$
\mathcal{F}f(\xi) : = \int_{\mathbb{R}}f(x)e^{-2\pi ix\xi}dx, \quad \xi \in \mathbb{R}
$$
and it extends to a unitary operator $\mathcal{F}: L^2(\mathbb{R)} \longrightarrow L^2(\mathbb{R})$. We say that a function $f$ in $L^2(\mathbb{R)}$ belongs to the Paley--Wiener space $PW_\omega(\mathbb{R})$ if the essential support of its Fourier transform is a subset of the interval $[-\omega,\omega]$.

Thanks to the Paley--Wiener theorem, we know that each $L^2$-class of function in $PW_\omega(\mathbb{R})$ has a class representative which can be extended to an entire function of exponential type $2\pi\omega$ (i.e., to an entire function $F$ satisfying $|F(z)| \leq \mathrm{A}e^{2\pi\omega|\mathfrak{Im}z|} $, for all $z \in \mathbb{C}$ and for some positive constant $A$).

Paley--Wiener spaces are essential in sampling theory because of Shannon's theorem.
\begin{theorem}
    \emph{(Shannon).} If $f \in PW_\omega(\mathbb{R})$, then
    \begin{equation}
       f = \sum_{j \in \mathbb{Z}}f\left(\frac{j}{2\omega}\right)\mathrm{sinc}\left(2\pi\omega\left(\cdot-\frac{j}{2\omega}\right)\right),
    \end{equation}
    where convergence is understood in the $L^2$-norm. Moreover, if we consider the representative $F$ of $f$ which can be extended to an entire function, we have
    \begin{equation}
    F(\xi) = \sum_{j \in \mathbb{Z}}F\left(\frac{j}{2\omega}\right)\mathrm{sinc}\left(2\pi\omega\left(\xi-\frac{j}{2\omega}\right)\right), \quad \xi \in \mathbb{R}
    \end{equation}
    and the series converges uniformly.
    \end{theorem}
Our goal is to provide an analogous result for a class of graphs with at most a countable set of vertices, like trees of bounded geometry and $\mathbb{Z}^n$-lattices. This result will be a generalization of the reconstruction theorem found in \cite{PW2} for finite graphs.

\subsection{Pesenson's approach in \cite{PW1} and main results}
We will now recall the main definitions and results from \cite{PW1} that are going to be used in our work, along with remarks on why its main reconstruction Theorem (3.1) fails for infinite-dimensional Paley--Wiener spaces.

We identify a graph $G$ with the couple $(V,E)$, where $V$ is the set of its vertices and $E$ is the set of its edges. We only take into account graphs $G$ for which the following assumptions hold:
\begin{enumerate}
    \item $G$ has no isolated vertices.
    \item $V$ is (at most) countable.
    \item $G$ has no loops (i.e. there are no vertices connected to themselves by an edge - notice that, in this setting, closed chains are acceptable).
    \item Multi-edges are not allowed (i.e. two vertices are connected at most by a single edge).
    \item Edges are uniformly weighted and undirected.
    \item Given $v \in V$, we denote by $d(v)$ the number of vertices connected to $v$ by an edge. We call this quantity the degree of $v$. We require the degree of the graph $d(G) := \sup_{v \in V}d(v)$ to be finite.
\end{enumerate}

In order to develop our sampling theory, we consider the space of square-summable functions $L^2(G)$ defined as follows:
$$
L^2(G) : = \left\{f: V \longrightarrow \mathbb{C} \, \middle| \sum_{v \in V}|f(v)|^2 < \infty\right\}.
$$
This space is endowed with the usual inner product:
$$
\left<f,g\right> = \sum_{v \in V}f(v)\overline{g(v)}.
$$
It follows immediately from this definition that we are considering the counting measure on the set of vertices $V$.

To introduce Paley--Wiener spaces on graphs, we need a notion of Fourier transform. Under our mild assumptions on $G$, no closed-form expression for a Fourier-type transform on $L^2(G)$ is available. However, we can overcome this issue by considering the properties of the discrete Laplacian operator $\mathcal{L}: L^2(G) \longrightarrow L^2(G)$. This operator is given by
\begin{equation}
\label{disclap}
  \mathcal{L} f(v)=\frac{1}{\sqrt{d(v)}} \sum_{v \sim u}\left(\frac{f(v)}{\sqrt{d(v)}}-\frac{f(u)}{\sqrt{d(u)}}\right), \quad f \in L^2(G), v\in V,  
\end{equation}
where $u \sim v$ means that $u$ is connected to $v$ by an edge. This definition was introduced in \cite{Chung}. Notice that the first assumption on the graph's structure, \say{G has no isolated nodes}, is fundamental to ensure that $d(v) \neq 0$ for each $v$ in $V$.

It can be shown that $\mathcal{L}$ is linear, bounded, positive and self-adjoint. Thus we can exploit the spectral theorem for bounded, linear, self-adjoint operators which is recorded below in the required form (see \cite{Hall} or \cite{Folland} for reference).
\begin{theorem}
\label{teo spett}
Let $\mathbf{H}$ be a Hilbert space. If $T: \mathbf{H} \longrightarrow \mathbf{H}$ is a bounded self-adjoint operator, then there exist a $\sigma$-finite measure $\mu$ on the spectrum $\sigma(T)$ of $T$, a direct integral 
\begin{equation}
 \mathbb{H}^{\oplus} := \int_{\sigma(A)}^{\oplus} \mathbf{H}_\tau d \mu(\tau)   
\end{equation}
and a unitary operator $U:\mathbf{H}\longrightarrow\mathbb{H}^{\oplus}$ such that
\begin{equation}
  [UTU^{-1}(s)](\tau) = \tau s(\tau), \quad \forall \tau \in \sigma(T), \forall s \in \mathbb{H}^{\oplus}.  
\end{equation}
\end{theorem}
Theorem~\ref{teo spett} states the existence of a unitary operator $U$ which can serve the same purpose as the Fourier transform in the definition of Paley--Wiener spaces on $\mathbb{R}$. Thus, we just need to fix the operator $U$ and we are ready to give a definition of Paley--Wiener spaces on graphs.
\begin{definition}
\label{def PW}
(Paley--Wiener space). Given $\omega \geq 0$, we say that a function $f$ in $L^{2}(G)$ belongs to the Paley--Wiener space $P W_\omega(G)$ if $Uf$ has essential support in $[0, \omega]$, where $U$ is a fixed unitary operator given by Theorem \ref{teo spett}.
\end{definition}
The observation that naturally arises from this definition is that, since $\mathcal{L}$ is positive, the lower bound of its spectrum is $0$, so the support of $Uf$ is not symmetric with respect to the origin as in the continuous setting. Moreover, since $\mathcal{L}$ is bounded, every function in $L^{2}(G)$ is \textit{bandlimited}. Thus the following relations hold:
\begin{equation}
\begin{split}
      &L^2(G)=P W_{\omega_{\mathrm{max}}}(G)=\bigcup_{\omega \in \sigma(\mathcal{L})} P W_\omega(G), \\
      &P W_{\omega_1}(G) \subseteq P W_{\omega_2}(G), \quad \omega_1<\omega_2, 
\end{split}
\end{equation}
where $\omega_{\mathrm{max}} : = \sup_{\omega \in \sigma(\mathcal{L})}(w)$. This is different from the continuous case because the whole $L^2(G)$ space can be seen as a Paley--Wiener space. This might lead us to think that it is possible to perform sampling on the whole $L^2(G)$. However, as we will point out later in this section, this is not the case.

The results in \cite{PW1} are obtained by means of frame theory, hence we now give a brief overview of this matter (for a more detailed discussion, see e.g. \cite{KG} or \cite{DS}).

A sequence $(e_j )_{j \in J}$ in a Hilbert space $\mathcal{H}$ is called a frame if there exist two constants $A,B>0$ such that for all $f \in \mathcal{H}$ we have
$$
A\|f\|^2_{\mathcal{H}}\leq\sum_{j\in J}|\left<f, e_j\right>|^2\leq B\|f\|^2_{\mathcal{H}}.
$$

The frame operator F is 
$$
Ff = \sum_{j\in J}\left<f, e_j\right>e_j, \quad f\in \mathcal{H}.
$$
It can be shown that the frame operator $F$ is invertible and that for all $f \in \mathcal{H}$ the following equality holds:
\begin{equation}
\label{feq}
   f = \sum_{j \in J}\left<f, e_j\right>F^{-1}e_j, 
\end{equation}
where the sequence $(F^{-1}e_j)_{j \in J}$ is referred to as the \textit{dual frame} of $(e_j)_{j \in J}$. This relation is at the core of our work, as we will now show.

The key point presented in \cite{PW1} is to exploit frame theory by finding a proper subset $W \subset V$ such that the following frame inequalities hold:
\begin{equation}
A\|f\|_2^2 \leq \sum_{v \in W}\left|\left\langle f, \theta_v\right\rangle\right|^2 \leq B\|f\|_2^2, \quad f\in PW_\omega(G),  
\end{equation}
where $\theta_v$ is the orthogonal projection on $PW_\omega(G)$ of the Dirac delta $\delta_v$. This would entail the existence of a dual frame $(\Theta_v)_{v\in W}$ such that, by equation \eqref{feq} for all $f$ in $PW_\omega(G)$:
$$
f(u) = \sum_{v \in W}\langle f,\theta_v\rangle\Theta_v(u) = \sum_{v \in W}f(v)\Theta_v(u),  \quad u \in V,
$$
thus recovering the results of Shannon's theorem, by which we can reconstruct a function from its values on a suitable subset of its domain. It is worth pointing out that in this case convergence is not uniform but unconditional, which means that for all $\epsilon >0$ there exists a finite subset $M_0(\epsilon) \subseteq W$ such that
$$
\left\|f - \sum_{v \in M}f(v)\Theta_v\right\|_2 < \epsilon
$$
for all finite subsets $M \supseteq M_0(\epsilon)$.

In order to find the desired frame inequalities, \cite{PW1} introduces the following crucial definitions of uniqueness and $\lambda$-sets:
\begin{definition}
\label{def uset}
(Uniqueness set). Let $\omega \in \sigma(\mathcal{L)}$. We will say that a subset $W\subseteq V$ is a uniqueness set for the space $PW_\omega(G)$ if for all $f,g \in PW_\omega(G)$, $\restr{f}{W}\equiv\restr{g}{W}$ implies that $ f\equiv g$.
\end{definition}
\begin{definition}
\label{def lset}
($\lambda$-set). We say that $S \subseteq V$ is a $\lambda$-set for $G$ if for all $\varphi \in L^2(S)$ the following Poincaré inequality holds: 
    \begin{equation}
    \label{lambdaset}
           \|\varphi\|_{L^2(S)}\leq \lambda \|\mathcal{L}\varphi\|_{2}.  
    \end{equation}
\end{definition}
The need for the notion of uniqueness set is intuitive: every sampling set must also be a uniqueness set. This definition also provides the reason why sampling on the whole $L^2(G)$ is not possible: suppose that there exists a proper subset $W \subset V$ that is a uniqueness set for $L^2(G)$. This means that there exists a vertex $v \in V$ that does not belong to $W$. This implies that the set $V\setminus\{v\}$ is also a uniqueness set for $L^2(G)$. Now, we take $f \in L^2(G)$ and we define $g$ such that $\restr{g}{V\setminus\{v\}}\equiv \restr{f}{V\setminus\{v\}}$ and $g(v) = f(v) +1$. This is a contradiction to our definition of uniqueness set, thus the whole $L^2(G)$ cannot be sampled.

It is perhaps less immediately clear why we also need the definition of $\lambda$-set: this is shown in the following theorem, proved in \cite{PW1}.
\begin{theorem}
\label{teoremerda}
    \emph{(\cite{PW1}, Theorem 3.2)}. If $S\subset V$ is a $\lambda$-set, then $W:=V\backslash S$ is a uniqueness set for all $PW_\omega(G)$ spaces with $\omega\lambda<1$.
\end{theorem}
This means that the strategy for finding uniqueness sets is by subtraction of $\lambda$-sets. Some results involving $\lambda$-sets shown in \cite{PW1} are collected in the following lemma.
\begin{lemma}
\label{lemma}
    For $S \subseteq V$, define $bS := \{v \in V\setminus S: v\sim u, u \in S\}$ and $\overline{S} := S \cup bS$. Then the following statements hold:
    \begin{enumerate}
        \item Every finite subset of $V$ is a $\lambda$-set.
        \item Let $S\subset V$ (finite or infinite) be such that for all $v \in S$ we have $\overline{\{v\}}\cap S = \{v\}$. Then $S$ is a $\lambda$-set with $\lambda = 1$.
        \item   Assume that $\{S_j\}_{j\in J}$ is a (finite or infinite) sequence of subsets $S_j \subset V$ such that the sets $\overline{S}_j$ are pairwise disjoint. If $S_j$ is a $\lambda_j$-set, then $S:=\bigcup_{j \in J}S_j$ is a $\lambda$-set, where $\lambda := \sup_{j\in J}\lambda_j$.
    \end{enumerate}
\end{lemma}
This leads us to the frame reconstruction theorem provided in \cite{PW1}, which claims the following:
\begin{theorem}
\label{teoPes}
    \emph{(\cite{PW1}, Theorem 3.1)}. A set of vertices $U \subset V$ is a uniqueness set for the space $P W_\omega(G)$ if and only if there exists a constant $C_\omega$ such that for any $f \in P W_\omega(G)$ the following discrete version of the Plancherel-Polya inequalities holds true:
\begin{equation}
\label{PesIneq}
    \left(\sum_{u \in U}|f(u)|^2\right)^{1 / 2} \leq\|f\|_{2} \leq C_\omega\left(\sum_{u \in U}|f(u)|^2\right)^{1 / 2}
\end{equation}
for all $f \in P W_\omega(G)$.
\end{theorem}
Let us make two remarks on the theorem.
\begin{enumerate}
    \item The chain of inequalities \eqref{PesIneq} is the frame chain of inequalities we are looking for: we just need to square it and then write $f(u)$ as $\left<f, \theta_u\right>$;
    \item If $U$ is a uniqueness set for $PW_\omega(G)$, then $$\|\cdot\|_U : PW_\omega(G) \ni f  \mapsto\left(\sum_{u \in U}|f(u)|^2\right)^{1/2} \in \mathbb{R}_+$$ is a norm for $PW_\omega(G)$: it is obviously a seminorm, but since $U$ is a uniqueness set, this means that if $\restr{f}{U}\equiv 0$, then $f \equiv 0$.
\end{enumerate}
Remark 2 is what is exploited in \cite{PW1} to prove Theorem \ref{teoPes}, by using the closed graph theorem applied to the identity operator between $(PW_\omega(G), \|\cdot\|_2)$ and $(PW_\omega(G), \|\cdot\|_U)$. However, this theorem only works for Banach spaces, and we do not know whether or not $(PW_\omega(G), \|\cdot\|_U)$ is complete. This is presumably the reason why the statement of Theorem \ref{teoPes} is corrected in the Erratum \cite{PW2}, by adding the crucial assumption that the space $PW_\omega(G)$ must be finite-dimensional. This assumption makes the proof trivial since all the norms defined on a finite-dimensional vector space are equivalent. However, this also makes the results of this theorem far less interesting: the most studied graphs in harmonic analysis (like homogeneous trees - see e.g. \cite{figa}- and $\mathbb{Z}^n$-lattices) all have infinite-dimensional Paley--Wiener spaces. Thus, our goal must be finding a sampling theorem that would allow for robust reconstruction in the infinite-dimensional scenario. In the next section, we are going to provide such a theorem. Before that, we will explicitly state the definition of sampling sets.

\begin{definition}
    (Sampling set). We say that the set $W \subseteq V$ is a sampling set for the Paley--Wiener space $PW_\omega(G)$ if the following chain of inequalities holds for some $c_\omega>0$:
    \begin{equation}
        \left(\sum_{v \in W}|f(v)|^2\right)^{1 / 2} \leq\|f\|_{2} \leq c_\omega\left(\sum_{v \in W}|f(v)|^2\right)^{1 / 2}.
    \end{equation}
    We also write
    $$
    \|f\|_{W} := \left(\sum_{v \in W}|f(v)|^2\right)^{1 / 2}.
    $$
\end{definition}

\section{Sampling theorems for infinite-dimensional Paley--Wiener spaces}
We start this section by proving a frame reconstruction sampling theorem which links sampling sets to $\lambda$-sets. 
\begin{theorem}
\label{belteorema}
    If $S$ is a $\lambda$-set and $W:= V\setminus S$, then
    \begin{equation}
       \|f\|_{W} \leq \|f\|_2 \leq c_\omega\|f\|_W 
    \end{equation}  
    for all $f$ in $PW_\omega(G)$ such that $\omega\lambda < 1$.
\end{theorem}
\begin{proof}
    The first inequality is immediate. As for the second one, we have
    $$
    \|f\|_2 = \|\restr{f}{W} + \restr{f}{S}\|_2 \leq \|\restr{f}{W}\|_2 + \|\restr{f}{S}\|_2.
    $$
    Now, obviously $\restr{f}{S}$ is a $L^2(S)$-function, and since $S$ is a $\lambda$-set, we have
    \begin{equation}
    \label{1}
       \|\restr{f}{S}\|_2 \leq \lambda \|\mathcal{L}\left(\restr{f}{S}\right)\|_2. 
    \end{equation}
    
    Since $\mathcal{L}$ is linear, we have
    $$
    \|\mathcal{L}\left(\restr{f}{S}\right)\|_2 =  \|\mathcal{L}f - \mathcal{L}\left(\restr{f}{W}\right)\|_2 \leq \|\mathcal{L}f\|_2 + \|\mathcal{L}\left(\restr{f}{W}\right)\|_2.
    $$
    Our aim now is to bound the sum above in terms of $\|f\|_2$ and $\|\restr{f}{W}\|_2$. As for $\|\mathcal{L}f\|_2$, we have:
    $$
               \|\mathcal{L}f\|_2^2 = \|U\mathcal{L}f\|_\oplus^2 = \int_{\sigma(\mathcal{L})}\|U\mathcal{L}f(\tau)\|^2_{\mathcal{H}_\tau}d\mu(\tau) = \int_{\sigma(\mathcal{L})}\|\tau Uf(\tau)\|^2_{\mathcal{H}_\tau}d\mu(\tau).
    $$
    Since $f \in PW_\omega(G)$, we have
    $$
    \begin{aligned}
     \int_{\sigma(\mathcal{L})}\|\tau Uf(\tau)\|^2_{\mathcal{H}_\tau}d\mu(\tau) &= \int_{[0, \omega]}\|\tau Uf(\tau)\|^2_{\mathcal{H}_\tau}d\mu(\tau) \\
    &\leq \omega^2\int_{[0, \omega]}\| Uf(\tau)\|^2_{\mathcal{H}_\tau}d\mu(\tau) = \omega^2\|Uf\|^2_\oplus = \omega^2\|f\|^2_2.   
    \end{aligned}
    $$
    This yields
    \begin{equation}
    \label{2}
    \|\mathcal{L}f\|_2 \leq \omega\|f\|_2.
    \end{equation}
    Now we only need to estimate $\|\mathcal{L}\left(\restr{f}{W}\right)\|_2$:
    \begin{equation}
    \label{3}
     \|\mathcal{L}\left(\restr{f}{W}\right)\|_2 \leq \|\mathcal{L}\|\|\restr{f}{W}\|_2 = \omega_\mathrm{max}\|\restr{f}{W}\|_2,
    \end{equation}
    where $\|\mathcal{L}\|$ is the norm of the Laplacian operator, and it equals $\omega_{\mathrm{max}}$.
    
    Combinig \eqref{1}, \eqref{2}, and \eqref{3} we have
    $$
    \|f\|_2 \leq \|f\|_W + \lambda\left(\omega \|f\|_2 + \omega_\mathrm{max}\|f\|_W\right).
    $$
    Hence
    $$
    \|f\|_2 \leq \frac{1 + \lambda \omega_\mathrm{max}}{1 - \lambda\omega}\|f\|_W
    $$
    and since $\lambda\omega < 1$, this concludes our proof.
\end{proof}
Observe that this theorem is a stronger version of Theorem \ref{teoremerda}: along with showing that the complement of a $\lambda$-set is a uniqueness set, it also shows that it is a sampling set.

This theorem solves our problem only partially, since it provides a proof for norm equivalence only if $W$ is the complement of a $\lambda$-set, and we do not know if the equivalence holds whenever $W$ is a uniqueness set which is not the complement of a $\lambda$-set. Indeed, we do not even know whether such uniqueness sets exist. However, the next theorem shows that every sampling set is the complement of a $\lambda$-set. This means that even if there exist uniqueness sets which are not complement of $\lambda$-sets, they are useless in terms of sampling results through frame inequalities.\\
It is also worth noting that in the continuous case it is not true that every uniqueness set is also a sampling set. Namely, for a uniqueness set $\{x_j\}_{j \in J}$ to be a sampling set, a separation condition $|x_j-x_k|\geq\delta$, $j \neq k$ should hold (see \cite{Polya}, or Remark 4 from \cite{PW1}).
\begin{theorem}
\label{iff}
    If $W$ is a sampling set for $PW_\omega(G)$, then its complement $S:= V\setminus W$ is a $\lambda$-set.
\end{theorem}

\begin{proof}
Consider $\varphi \in L^2(S)$ and denote by $\mathbb{P}_\omega$ the projection of $L^2(G)$ onto $PW_\omega(G)$. We define
$$
f : = \mathbb{P}_\omega\varphi.
$$
This gives us the following equalities:
$$
\|f - \varphi\|_2^2 = \|U(f - \varphi)\|_\oplus^2 = \int_{\sigma(\mathcal{L)}}\|U(f - \varphi)(\tau)\|^2_{\mathcal{H}_\tau}d\mu(\tau) = \int_{\sigma(\mathcal{L})\setminus [0, \omega]}\|U\varphi(\tau)\|^2_{\mathcal{H}_\tau}d\mu(\tau). 
$$
Since for all $\tau$ in $\sigma(\mathcal{L})\setminus[0, \omega]$ we have $\tau/\omega>1$, we get
$$
\begin{aligned}
    \int_{\sigma(\mathcal{L})\setminus [0, \omega]}\|U\varphi(\tau)\|^2_{\mathcal{H}_\tau}d\mu(\tau) &\leq  \int_{\sigma(\mathcal{L})\setminus [0, \omega]}\frac{\tau^2}{\omega^2}\|U\varphi(\tau)\|^2_{\mathcal{H}_\tau}d\mu(\tau)
    = \frac{1}{\omega^2}\int_{\sigma(\mathcal{L})\setminus [0, \omega]}\|\tau U\varphi(\tau)\|^2_{\mathcal{H}_\tau}d\mu(\tau) \\\\
    &= \frac{1}{\omega^2}\int_{\sigma(\mathcal{L})\setminus [0, \omega]}\|U\mathcal{L}\varphi(\tau)\|^2_{\mathcal{H}_\tau}d\mu(\tau) \leq \frac{1}{\omega^2}\int_{\sigma(\mathcal{L})}\|U\mathcal{L}\varphi(\tau)\|^2_{\mathcal{H}_\tau}d\mu(\tau)\\\\
    & = \frac{1}{\omega^2}\|U\mathcal{L}\varphi\|_\oplus^2 = \frac{1}{\omega^2}\|\mathcal{L}\varphi\|_2^2.
\end{aligned}
$$
This means that $$\|f - \varphi\|_2 \leq \frac{1}{\omega}\|\mathcal{L}\varphi\|_2. $$
Since $W$ is a sampling set for $PW_\omega(G)$, this entails that there exists a constant $c_\omega >0$ such that
$$
\|f\|_2 \leq c_\omega\|f\|_W, \quad f \in PW_{\omega}(G).
$$
Now, since $\mathrm{supp}\varphi \subseteq S$, we have 
$$
\|f\|_W = \|f - \varphi\|_W \leq \|f - \varphi\|_2.
$$
All in all, we get the following chain of inequalities:
$$
\begin{aligned}
    \|\varphi\|_2 &\leq \| \varphi - f\|_2 + \|f\|_2 \leq \frac{1}{\omega}\|\mathcal{L}\varphi\|_2 + c_\omega\|f\|_W \\
    &\leq \frac{1}{\omega}\|\mathcal{L}\varphi\|_2 + c_\omega\|\varphi - f\|_2 \leq \frac{1 + c_\omega}{\omega}\|\mathcal{L}\varphi\|_2.
\end{aligned}
$$
Therefore $S : = V \setminus W$ is a $\lambda$-set, with 
$$\lambda \leq \frac{1 + c_\omega}{\omega}.$$
\end{proof}
Theorem \ref{iff} allows us to consider only the complement of $\lambda$-sets in order to find useful frame-reconstruction sampling sets for a Paley--Wiener space $PW_\omega(G)$. However, this result is not enough to establish a full characterisation of such subsets, since for a fixed $\omega >0$, we have that $\lambda\omega \leq 1+ c_\omega$, while Theorem \ref{belteorema} gave us the condition $\lambda\omega<1$. In order to dig deeper into this connection and find out whether or not it would be possible to characterise sampling sets via the $\lambda$-constant of their complement, one should start by trying to characterise $c_\omega$ in terms of the geometric constraints of the graph. Some work has already been done in this sense, and some results where the sampling threshold depends on geometrical features of the graph have been published (see e.g. \cite{FP}, \cite{PP1}). However, since these results are not sharp, they do not provide a sharp $\lambda$-constant for their complements, so the issue still persists. It seems reasonable that the assumption of finite-geometry (or, in a more general setting, the boundedness of the weight function) would play a crucial role into finding the desired characterisation. This problem is however very difficult to tackle, and it is not precisely on the scope of this work, so we will leave this question open.\\
Note that the proof of Theorem \ref{iff} nowhere uses the explicit form \ref{disclap} of the discrete Laplacian; consequently, the result holds for any bounded, linear, positive operator playing its role, such as the operators in \cite{FP}.


\section{Remarks on \cite{FP}}
In this section, we will discuss \cite{FP} and compare its results with the ones provided in Section 3. We will start by giving a brief overview of its framework, which is more general. We will then highlight the differences in the presented results once they are adapted to our framework. Finally, we will give an example on $\mathbb{Z}$ of a sampling set which can be found either by our results or by the sampling theorem in \cite{FP}. We will show that with our method the sampling threshold will be higher than the one given by \cite{FP}, thus allowing for sampling on a wider range of Paley--Wiener spaces.

\subsection{Framework}
The assumptions on the graph $G$ coincide with ours, with the exception of uniformly weighted graphs. This means that the finite-degree assumption can be relaxed: the authors consider a non-negative weight function $w: V\times V\longrightarrow \mathbb{R}_+$ and they ask that $\sum_{v \in V}w(u,v) < \infty$ for all $u \in V$.\\
They also take into account graphs equipped with positive measures $\nu: V\longrightarrow (0,\infty)$, which are not necessarily the counting measure. The space of $L^2$-functions that they consider is then defined as the space of all functions $f: V \longrightarrow\mathbb{C}$ such that
$$
\|f\|^2_2 = \sum_{v \in V}|f(v)|^2\nu(v) < \infty.
$$
Given this framework, the discrete Laplacian operator $\Delta$ is defined as follows:
$$
\Delta f(v) = \sum_{u \sim v}(f(v) - f(u))w(v,u).
$$
The first essential remark we need to make is that although $\Delta$ is bounded, linear, and positive, it is not equivalent to $\mathcal{L}$ whenever the degree of $G$ is not homogeneous and the graph is not uniformly weighted, with weight $w(u,v) = 1/d(G)$ for all $(u, v)$ in $E$. So, the results found in \cite{FP} do not actually translate to those found here except for the case we just stated. Still, we can consider uniformly weighted graphs such as $\mathbb{Z}^n$-lattices or homogeneous trees, equipped with the counting measure and uniform weight $1/d(G)$. We will discuss the differences in these scenarios in the next subsection.

\subsection{Differences in homogeneous, uniformly weighted graphs}
The first argument we need to make in order to highlight the differences between our work and \cite{FP} is the methodology by which it is aimed to find sampling sets. Our aim is to focus on the interplay between $\lambda$-sets and sampling sets; particularly, we want to recover the original results presented in \cite{PW1}: this would allow us to find sampling sets by subtraction of $\lambda$-sets. We could then make use of the various techniques presented in \cite{PW1} for finding $\lambda$-sets and by that be able to recover the associated sampling sets. On the other hand, in \cite{FP} an explicit geometric construction, involving so called \say{admissible partitions} of the set of vertices, is presented. By building such partitions, one could find that their initial set $S_0 \subset V$ is actually a sampling set for Paley--Wiener spaces which obey to a geometric costraint depending on the partition. This is shown in Theorem 1.4, Corollary 1.9 and inequality (6.1) from \cite{FP}. Corollary 1.10 then shows that if we remove the sampling set $S_0$, we get that its complement $V\setminus S_0$ is a $\lambda$-set. This result would be trivial in our work since we proved that every sampling set, independently of how it is constructed, is the complement of a $\lambda$-set.\\
In the next subsection we are going to see an example that could be found both through the techniques presented in \cite{FP} and the ones presented in \cite{PW1}. We will show that, in light of Theorem \ref{belteorema}, the latter construction provides a more relaxed constraint on the range of Paley--Wiener spaces that can be sampled.

\subsection{Example: comparison on the line $\mathbb{Z}$}
We consider the line graph $\mathbb{Z}$, where every vertex is an integer. The discrete Laplacian operator's spectrum $\sigma(\mathcal{L})$ is the interval $[0,2]$. We consider the sampling set $W$ built in this way: 
$$
W : = k\mathbb{Z}\cup \left\{k\mathbb{Z}+ \{1\}\right\}, \quad k \in 2\mathbb{N}\setminus\{2\}.
$$
Now we perform the due computations to verify which Paley--Wiener spaces can be sampled on this set, starting from the ones provided by \cite{PW1}:
\begin{itemize}
    \item Lemma 5.3 from \cite{PW1} states that every finite sequence of consecutive vertices of lenght N is a $\lambda$-set with
    \begin{equation}
     \label{lambda}
     \lambda = \frac{1}{2}sin^{-2}\left(\frac{\pi}{2N + 2}\right).
    \end{equation}
    Moreover, point $3.$ of Lemma \ref{lemma}, states that if we have a collection of $\lambda$-sets such that their closures are pairwise disjoint, then their union is a $\lambda$-set where the constant $\lambda$ is the supremum among all the constants belonging to the sets making up the collection. In our case, since all constants are the same, we get that $V\setminus W$ is a $\lambda$-set, with $\lambda$ given by \ref{lambda}. By Theorem \ref{belteorema}, we get that $W$ is a viable sampling set for all $PW_\omega(\mathbb{Z)}$ spaces such that
    $$
    \omega < 2sin^2\left(\frac{\pi}{2(k-2) + 2}\right).
    $$
    \item Now we take into account the results presented in \cite{FP}. An explicit computation of the geometric threshold for the sampling set $r\mathbb{Z}$, with $r$ odd, is performed in section 7. Our sampling set is larger, so the Paley--Wiener spaces that can be sampled on $r\mathbb{Z}$ must also be sampled on $W$. By the same reasoning, it is not given for granted that there are no other Paley--Wiener spaces that might be sampled on $W$. However, the most efficient partition that could be found by \cite{FP} leads to the same calculations of the sampling threshold for $r\mathbb{Z}$, $r$ odd. In \cite{FP}, such a partition is built as follows:
    $$
    S_m = \{\pm m\} + r\mathbb{Z}. 
    $$
    This is the best way to optimize the sampling threshold, which is given by inequality (6.1) from \cite{FP} i.e.
    \begin{equation}
    \label{threshold}
      \omega < \frac{1}{2}\left(\sum_{m=1}^n \sum_{j=1}^m \frac{1}{K_{j-1}} \prod_{i=j}^{m-1} \frac{D_i}{K_i}\right)^{-1},  
    \end{equation}
    where
    $$D_j := \sup_{v \in S_j}w_{S_{j+1}}(v);$$
    $$K_j := \inf_{v \in S_{j+1}}w_{S_{j}}(v);$$
    $$w_{S_j}(v) : = \sum_{u \in S_j}w(u,v).$$
    Now we consider a partition made in the following way:
    $$
    S_j = \{ mk - j \mid m \in \mathbb{Z} \} \cup \{ mk + j + 1 \mid m \in \mathbb{Z} \}, \quad j = 0, \dots, \frac{k-2}{2}.
    $$
    It is clear that $W = S_0$. The partition is schematised in Figure \ref{f1}, for the special case $k = 6$.
    \begin{center}
        \begin{figure}
            \centering
            \includegraphics[scale=0.25]{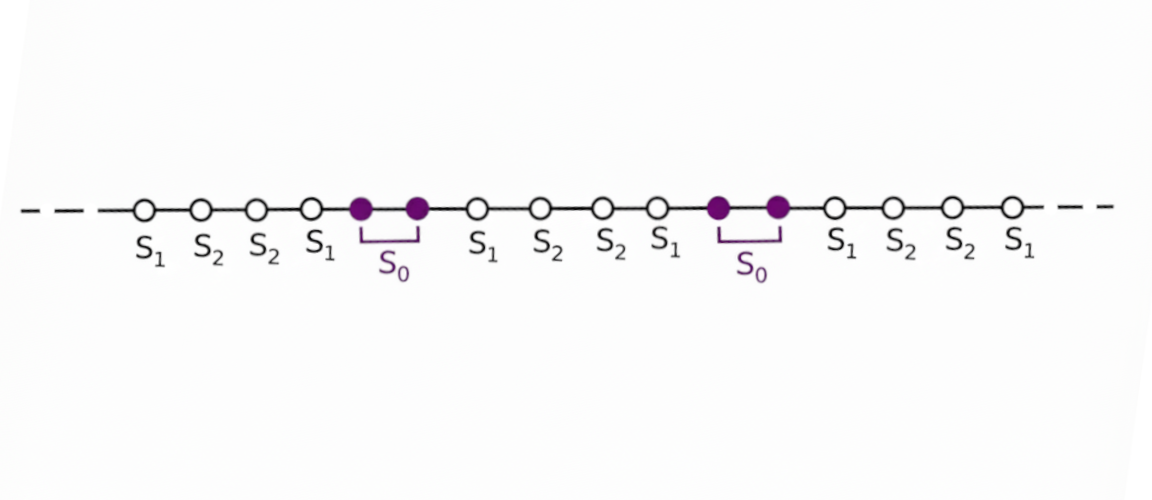}
            \caption{Admissible partition for $k=6$.}
            \label{f1}
        \end{figure}
    \end{center}
    It is straightforward to notice that this kind of partition leads to the same geometric calculations found in \cite{FP} for the sampling set $r\mathbb{Z}$, $r$ odd. The reasoning behind this choice is keeping an even number of vertices between the elements of the sampling set. Then, we can build a partition where the $D_j$'s are minimised. The $K_j$'s must be exactly $1$ for partitions on $\mathbb{Z}$ that have more than two elements, so there is no other way to improve the threshold. In the end, this choice allows to maximise the right hand side of \eqref{threshold}. The fact that a larger sampling set still yields the same geometric threshold indicates that the results provided in \cite{FP} are not always sharp (as the authors themselves recognise in the discussion of Corollary 7.2).\\
    The explicit geometric threshold is given by the first equality of (7.3) in \cite{FP}, together with the bound \eqref{threshold}, which results in
    \begin{equation}
       \omega < \frac{1}{n(n+1)}, \qquad n= \frac{k-2}{2}. 
    \end{equation}
    \item Now we just need to compare the two thresholds. A direct computation shows 
    $$
    a_n = 2sin^2\left(\frac{\pi}{4n+2}\right) - \frac{1}{n(n+1)}, \quad n\in \mathbb{N}\setminus\{0\}
    $$
    is a non-negative sequence that goes to $0$ as $n\rightarrow \infty$ (it is 0 only at $a_1 = 2sin^2\left(\frac{\pi}{6}\right) - \frac{1}{2}= 0$). It follows that our threshold is sharper for any $k \in 2\mathbb{N}\setminus\{2,4\}$ and it is the same for $k = 4$.
\end{itemize}

\section{Example: spheres on homogeneous trees}
In this section, we are going to look into some examples of sampling sets on homogeneous trees that can be found through \ref{belteorema}. In \cite{PW1}, some examples of this kind are given; however, the reconstruction theorem from \cite{PW1} was not to be considered valid since it was later corrected in \cite{PW2} and adapted to finite-dimensional Paley--Wiener spaces. In light of \ref{belteorema} we can now recover those examples and be sure of their validity. Since the examples found in \cite{PW1} are almost trivial, in the next subsection we will try to find some more notable examples.

\subsection{Framework}
We consider a homogeneous tree, i.e. a denumerable, connected graph with no cycles and homogeneous degree $d(G) = q+1$, $q \geq 2$. Again, we ask for uniformly weighted edges and we adopt the counting measure on $V$. We fix a vertex of the tree and we refer to it as its \say{root}. We then indicate with $l_m$ the sphere of radius $m$ with respect to the root, with $m \in \mathbb{N}$. We will say that a vertex belongs to $l_m$ if there are $m$ edges \say{separating} it from the root (so the root belongs to $l_0$, its neighbours to $l_1$, their neighbours to $l_2$ and so on).\\
A very deep and rich theory has been developed in \cite{figa} for this kind of trees. This theory is specialised to radial functions, i.e. $L^2$-functions such that they take the same value on $l_m$. With a slight abuse of notation, we will denote the space of radial functions on the homogeneous tree $G$ with $L^2(G)$. In the same fashion, we will use $PW_\omega(G)$ to denote spherical Paley--Wiener spaces defined as 
$$
PW_\omega(G): =  \left\{f \in L^2(G): \mathrm{supp}(Uf)\subseteq[0, \omega]\right\}.
$$
It is known from \cite{figa} that the Laplacian operator on radial functions has spectrum
\begin{equation}
\label{homspect}
    \sigma(\mathcal{L}) = [1- \rho(q), 1+ \rho(q)], \qquad \rho(q) = \frac{2 q^{1/2}}{q+1}.
\end{equation}

In the next subsection, we are going to exploit this fact together with the properties of radial functions to study sampling sets $W$ of the kind $$W = \bigcup_{m \in \mathbb{N}}l_{km}\cup l_{km+1},\quad k \in \mathbb{N}, \quad k>2.$$

\subsection{Explicit computations for $W$}
We first compute the explicit $\lambda$-constant for $S := V\setminus W$ and then explain why we chose these kind of sets. It is immediate to notice that the discrete Laplacian operator on homogeneous graphs can be rewritten in the matrix form:
\begin{equation}
    \label{hom lap}
    \mathcal{L} = I - \frac{1}{d(G)}A,
\end{equation}
where $A$ is the adjacency matrix of $G$, i.e. the symmetric matrix such that $A_{u,v}=1$ if $(u,v)$ belongs to $E$, and it is $0$ otherwise. We fix $k \in \mathbb{N}$, $k>2$, and we consider 
$$W= \bigcup_{m \in \mathbb{N}}l_{km}\cup l_{km+1}.$$
We take a fixed $m$ and we denote by $v_{km+1}$ one of the vertices of $l_{km+1}$. We then consider the $q$ sons of $v_{km+1}$ as roots of $q$-ary finite trees of depth $k-2$. This operation is schematised in Figure \ref{f2}. We denote these kind of finite sub-trees with $T_j$. We now consider a function $\varphi_j$ supported on $T_j$. It is clear that
$$
\langle \varphi_j, \mathcal{L}\varphi_j \rangle = \sum_{v \in T_j}\varphi_j(v)\overline{\mathcal{L}\varphi_j(v)} = \sum_{v \in T_j}\varphi_j(v)\overline{\mathcal{L}_j\varphi_j(v)},
$$
where $\mathcal{L}_j$ is the discrete Laplacian operator restricted to the subgraph identified by $T_j$. The last inequality holds because only vertices of $T_j$ are summed and for such vertices the two operators agree. We can rewrite $\mathcal{L}_j$ as
\begin{equation}
    \mathcal{L}_j = I_j -\frac{1}{q+1}A_j,
\end{equation}
where $I_j$ and $A_j$ are respectively the identity matrix and the adjacency matrix associated with $T_j$. This means that we can calculate the eigenvalues of $\mathcal{L}_j$ by knowing the eigenvalues of $A_j$. By Rayleigh's quotient plus the Cauchy-Schwarz inequality we obtain
$$
\eta_{\mathrm{min}}\|\varphi_j\|^2_2 \leq |\langle\varphi_j, \mathcal{L}_j\varphi_j \rangle| \leq \|\varphi_j\|_2\|\mathcal{L}\varphi_j\|_2,
$$
\begin{center}
   \begin{figure}
    \centering
    \includegraphics[scale = 0.20]{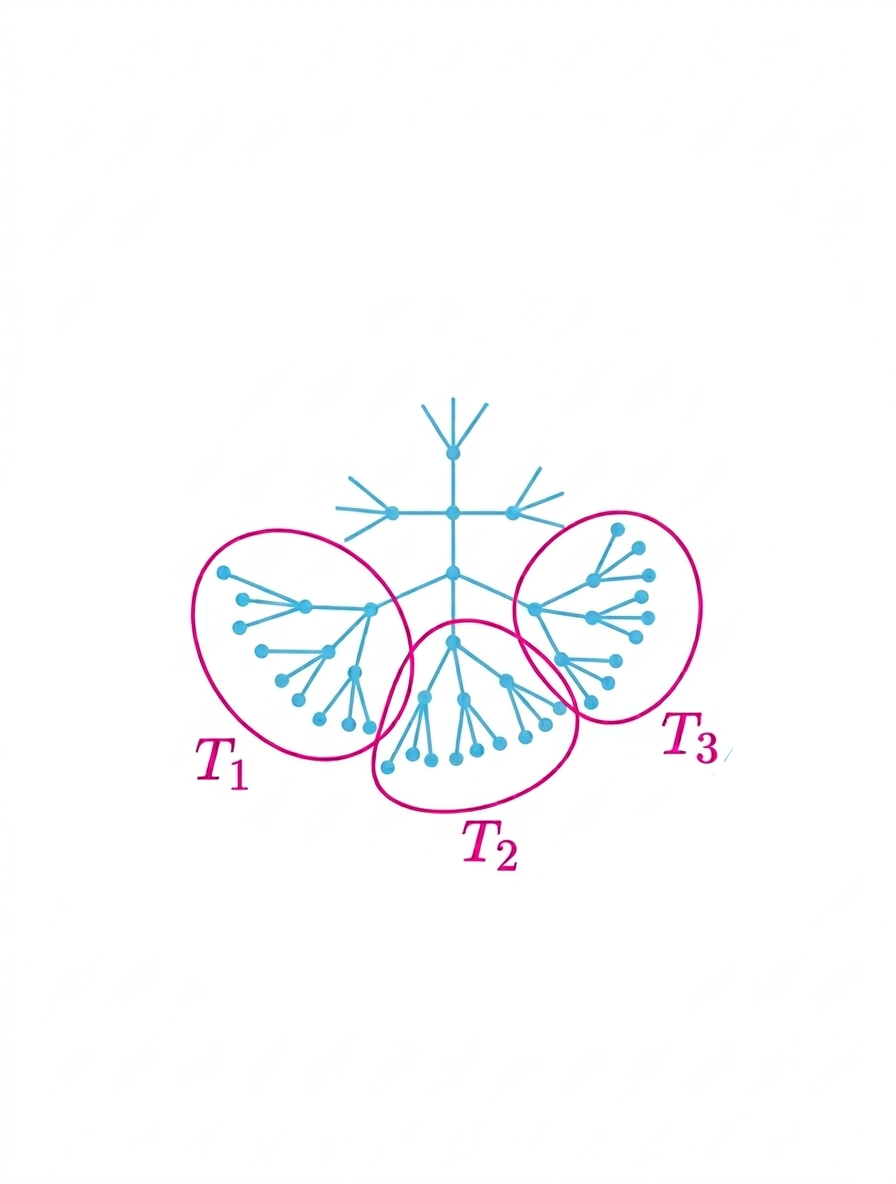} 
    \caption{Subtrees of depth $k-2 = 3$.}
    \label{f2}
\end{figure} 
\end{center}
where $\eta_{\mathrm{min}}$ is the minimal eigenvalue of $\mathcal{L}_j$. From Theorem 1 in \cite{TreeSpec}, we know that the set of the eigenvalues of the adjacency matrix of a q-ary finite tree of depth $k-2$ is
$$
\bigcup_{2 \leq m \leq k}\left\{2\sqrt{q}cos\left(\frac{l\pi}{m}\right): 1 \leq l \leq m \right\}.
$$
It is clear that the maximum eigenvalue of the adjacency matrix is $\mu_{\mathrm{max}}: = 2\sqrt{q}cos\left(\frac{\pi}{k}\right)$, so the minimal eigenvalue of $\mathcal{L}_j$ is given by
$$
\eta_{\mathrm{min}} = 1- \frac{1}{q+1}\mu_{max} = 1- \frac{2\sqrt{q}}{q+1}cos\left(\frac{\pi}{k}\right) > 0.
$$
This means that, for each $\varphi_j$ supported on $T_j$, we have
\begin{equation}
    \label{phi_est}
    \|\varphi_j\|_2 \leq \frac{1}{\eta_{\mathrm{min}}}\|\mathcal{L}\varphi_j\|_2.
\end{equation}
Now we consider $S = V\setminus W$ and we compute its $\lambda$-constant. Of course, $S= \bigcup_{j \in \mathbb{N}}T_j$. Since the $T_j$'s are disjointed, for all $\varphi$ in $L^2(S)$ we have
$$
\|\varphi\|^2_2 = \sum_{j \in \mathbb{N}}\|\varphi_j\|^2_2 \leq \sum_{j \in \mathbb{N}}\frac{1}{\eta_{\mathrm{min}}^2}\|\mathcal{L}\varphi_j\|^2_2 = \frac{1}{\eta_{\mathrm{min}}^2}\sum_{j \in \mathbb{N}}\|\mathcal{L}\varphi_j\|^2_2.
$$
Now we have to estimate $\sum_{j \in \mathbb{N}}\|\mathcal{L}\varphi_j\|^2_2$ in terms of $\|\mathcal{L}\varphi\|^2_2$. This time, since the $T_j$'s closures are not pairwise disjoint, we cannot appeal to Lemma \ref{lemma}. However, we can exploit the radial symmetry of $\varphi$: If we fix a vertex on $l_{km+1}$, and we denote it with $v_{km+1}$, we have that 
$$
\sum_{j \in \mathbb{N}}|\mathcal{L}\varphi_j(v_{km+1})|^2 = \frac{q}{(q+1)^2}|\varphi(v_{km+2})|^2 = \frac{1}{q}|\mathcal{L}\varphi(v_{km+1})|^2,
$$
moreover, for any $v$ in $V\setminus\bigcup_{m \in \mathbb{N}}l_{km+1}$, we have $$
\sum_{j \in \mathbb{N}}|\mathcal{L}\varphi_j(v)|^2=|\mathcal{L}\varphi(v)|^2.
$$
Thus $\sum_{j \in \mathbb{N}}\|\mathcal{L}_j\varphi_j\|^2_2 \leq \|\mathcal{L}\varphi\|^2_2$. In the end we get
\begin{equation}
    \|\varphi\|_2 \leq \frac{1}{\eta_{\mathrm{min}}}\|\mathcal{L}\varphi\|_2, \qquad \varphi \in L^2(S).
\end{equation}
By Theorem \ref{belteorema}, we have that $W$ is a sampling set for all radial Paley--Wiener spaces $PW_\omega(G)$ such that
$$
\omega < \eta_{\mathrm{min}}.
$$
Since the spectrum of $\mathcal{L}$ is the interval given by \eqref{homspect}, the sampling range is given by
$$
\omega \in \left[1 - \frac{2\sqrt{q}}{q+1},1- \frac{2\sqrt{q}}{q+1}cos\left(\frac{\pi}{k}\right)\right).
$$
As one might expect, as $k$ increases (i.e. the distance between elements of the sampling set increases), the sampling range gets narrower.
\subsection{Final remarks on the choice of the sampling set}
The choice of a sampling set built up by layers of two consecutive spheres was motivated by two main reasons (other than trying to make it as sparse as possible): first, since we are dealing with radial functions, it makes sense to sample on the whole sphere, since such functions take the same value at each $l_m$. Second, if we only considered the sampling set made up by separated spheres $\bigcup_{m \in \mathbb{N}}l_{km}$, $k >1$, we would not have been able to estimate $\sum_{j \in \mathbb{N}}\|\mathcal{L}_j\varphi_j\|_2$ in terms of $\|\mathcal{L}\varphi\|_2$, since for a fixed vertex on $l_{km}$ denoted with $v_{km}$, we would have had that 
$$
|\mathcal{L}\varphi_j(v_{km})| = \left|\frac{1}{q+1}\varphi_j(v_{km+1})\right|
$$ 
while 
$$
|\mathcal{L}\varphi(v_{km})|=\left|\frac{q}{q+1}\varphi_j(v_{km+1}) + \frac{1}{q+1}\varphi(v_{km-1})\right|.
$$ 
This also means that our argument is valid if we remove the sphere $l_1$ from our sampling set. 

An example for the case $k =3$ is also given in \cite{PW1}: in this case, the finite sub-trees making up $S$ are just single nodes, so we can exploit a sharper threshold given by
$$
\|\varphi\|_2 \leq \left(1 + \frac{q}{(q+1)^2}\right)^{-1/2}\|\mathcal{L}\varphi\|_2, \qquad \varphi \in L^2(l_m).
$$
By Lemma \ref{lemma}, $W$ would then be a viable sampling set for any $PW_\omega(G)$ space such that $$\omega < \left(1 + \frac{q}{(q+1)^2}\right)^{1/2},$$ which is of course a higher threshold than the one provided by our calculations. 
\section*{Acknowledgements}
This work was part of my MS Thesis at the University of Genoa. I deeply thank my supervisor, Professor Filippo De Mari, for guiding me throughout the development of this project and for giving me valuable insights during our many conversations. 

\bibliographystyle{amsplain}

\end{document}